\documentclass[12pt, letterpaper]{amsart}

\usepackage[pagebackref,hypertexnames=false, colorlinks, citecolor=red, linkcolor=red]{hyperref}
\usepackage{amssymb}
\usepackage{mathrsfs}

	\normalbaselines						  %

\newcommand{\R}{{\mathbb  R}}

\newcommand{\N}{{\mathbb  N}}

\newcommand{\fdot}{\,\cdot\,}

\newcommand{\wt}{\widetilde}

\newcommand{\cL}{\mathcal{L}}
\newcommand{\cP}{\mathcal{P}}

\newcommand{\cM}{\mathcal{M}}

\DeclareMathOperator{\spa}{span}

\catcode`@=11
\chardef\mathlig@atcode\count255

\def\actively#1#2{\begingroup\uccode`\~=`#2\relax\uppercase{\endgroup#1~}}
\def\mathlig@gobble{\afterassignment\mathlig@next@cmd\let\mathlig@next= }
\def\mathlig@delim{\mathlig@delim}
\def\mathlig@defcs#1{\expandafter\def\csname#1\endcsname}
\def\mathlig@let@cs#1#2{\expandafter\let\expandafter#1\csname#2\endcsname}
\def\mathlig@appendcs#1#2{\expandafter\edef\csname#1\endcsname{\csname#1\endcsname#2}}

\def\mathlig#1#2{\mathlig@checklig#1\mathlig@end\mathlig@defcs{mathlig@back@#1}{#2}\ignorespaces}

\def\mathlig@checklig#1#2\mathlig@end{%
 \expandafter\ifx\csname mathlig@forw@#1\endcsname\relax
 \expandafter\mathchardef\csname mathlig@back@#1\endcsname=\mathcode`#1%
 \mathcode`#1"8000\actively\def#1{\csname mathlig@look@#1\endcsname}%
 \mathlig@dolig#1\mathlig@delim
\fi
\mathlig@checksuffix#1#2\mathlig@end
}

\def\mathlig@checksuffix#1#2\mathlig@end{%
\ifx\mathlig@delim#2\mathlig@delim\relax\else\mathlig@checksuffix@{#1}#2\mathlig@end\fi
}
\def\mathlig@checksuffix@#1#2#3\mathlig@end{%
\expandafter\ifx\csname mathlig@forw@#1#2\endcsname\relax\mathlig@dosuffix{#1}{#2}\fi
\mathlig@checksuffix{#1#2}#3\mathlig@end
}

\def\mathlig@dosuffix#1#2{%
\mathlig@appendcs{mathlig@toks@#1}{#2}%
\mathlig@dolig{#1}{#2}\mathlig@delim
}

\def\mathlig@dolig#1#2\mathlig@delim{%
 \mathlig@defcs{mathlig@look@#1#2}{%
 \mathlig@let@cs\mathlig@next{mathlig@forw@#1#2}\futurelet\mathlig@next@tok\mathlig@next}%
 \mathlig@defcs{mathlig@forw@#1#2}{%
  \mathlig@let@cs\mathlig@next{mathlig@back@#1#2}%
  \mathlig@let@cs\checker{mathlig@chck@#1#2}%
  \mathlig@let@cs\mathligtoks{mathlig@toks@#1#2}%
  \expandafter\ifx\expandafter\mathlig@delim\mathligtoks\mathlig@delim\relax\else
  \expandafter\checker\mathligtoks\mathlig@delim\fi
  \mathlig@next
 }%
 \mathlig@defcs{mathlig@toks@#1#2}{}%
 \mathlig@defcs{mathlig@chck@#1#2}##1##2\mathlig@delim{%
  \ifx\mathlig@next@tok##1%
   \mathlig@let@cs\mathlig@next@cmd{mathlig@look@#1#2##1}\let\mathlig@next\mathlig@gobble
  \fi
  \ifx\mathlig@delim##2\mathlig@delim\relax\else
   \csname mathlig@chck@#1#2\endcsname##2\mathlig@delim
  \fi
 }%
 \ifx\mathlig@delim#2\mathlig@delim\else
  \mathlig@defcs{mathlig@back@#1#2}{\csname mathlig@back@#1\endcsname #2}%
 \fi
}%

\catcode`@\mathlig@atcode

\mathchardef\ordinarycolon\mathcode`\:
\def\vcentcolon{\mathrel{\mathop\ordinarycolon}}
\mathlig{:=}{\vcentcolon=}
\mathlig{::=}{\vcentcolon\vcentcolon=}

\numberwithin{equation}{section}

\theoremstyle{plain}
\newtheorem{theo}{Theorem}[section]
\newtheorem{cor}[theo]{Corollary}
\newtheorem{lem}[theo]{Lemma}

\theoremstyle{definition}

\theoremstyle{remark}
\newtheorem*{ex*}{Example}
\theoremstyle{remark}
\newtheorem*{exs*}{Examples}
\theoremstyle{remark}
\newtheorem*{rem*}{Remark}
\newtheorem{rem}[theo]{Remark}
\newtheorem*{rems*}{Remarks}

\title[Moment Representations]{Moment Representations of the Exceptional $X_1$-Laguerre Orthogonal Polynomials}
\author{Constanze~Liaw}
\address{CASPER and Department of Mathematics, Baylor University, One Bear Place \#97328,      
 Waco, TX  76798, USA}
\email{Constanze$\underline{\,\,\,}$Liaw@baylor.edu}
\urladdr{http://sites.baylor.edu/constanze$\underline{\,\,\,}$liaw/}

\author{John~Osborn}
\address{Department of Mathematics, Baylor University, One Bear Place \#97328,      
 Waco, TX  76798, USA}
\email{John$\underline{\,\,\,}$Osborn@baylor.edu}

\begin{document}
\date{\today}

\begin{abstract}
Exceptional orthogonal Laguerre polynomials can be viewed as an extension of the classical Laguerre polynomials per excluding polynomials of certain order(s) from being eigenfunctions for the corresponding exceptional differential operator. We are interested in the (so-called) Type I $X_1$-Laguerre polynomial sequence $\{L_n^\alpha\}_{n=1}^\infty$, $\deg p_n = n$ and $\alpha>0$, where the constant polynomial is omitted.

We derive two representations for the polynomials in terms of moments by using determinants. The first representation in terms of the canonical moments is rather cumbersome. We introduce adjusted moments and find a second, more elegant formula. We deduce a recursion formula for the moments and the adjusted ones. The adjusted moments are also expressed via a generating function. We observe a certain detachedness of the first two moments from the others.
\end{abstract}

\maketitle

\section{Introduction}
Exceptional orthogonal polynomials were originally discovered as exact solutions to certain models in quantum mechanics. Those models include the supersymmetric setting, the Fokker--Planck, as well as the Dirac equations. See \cite{Dutta-Roy, grandati11, Ho, Midya-Roy, Tanaka} for the connections to physics and mathematical physics. In recent years, the field has enjoyed much further attention from both the mathematics and the physics community; see e.g.~\cite{Atia-Littlejohn-Stewart, DuranJAT2014, KMUG, KMUG1, GUKM10, GUMM-Asymptotics, HoSasakiZeros2012, LLMS, Liaw-Littlejohn-Stewart, Odake-Sasaki, Quesne} and the references therein.

We focus on the exceptional $X_1$-Laguerre polynomials from the perspective of the seminal paper by Gom\'ez-Ullate--Kamran--Milson \cite{KMUG}. This polynomial sequence, denoted by $\{L_n^\alpha(x)\}_{n\in\N}$, $\alpha>0$, is orthogonal on $[0,\infty)$ with respect to the $X_1$-Laguerre weight
\[
W^{\alpha}(x)=\frac{x^{\alpha}e^{-x}}{(x+\alpha)^{2}%
}.
\]
The polynomials are complete in $L^2([0,\infty); W^\alpha)$ even though there is no degree $0$ polynomial. Further, they are the eigenfunctions of the \emph{exceptional $X_1$-Laguerre differential expression}
    \begin{equation} \label{eq:diff}
     \ell^\alpha[y] = -xy'' + \left(\frac{x-\alpha}{x+\alpha}\right) \left[(x+\alpha+1)y'-y\right].
    \end{equation}
The spectral analysis of the polynomial system and a rigorous definition of a self-adjoint operator corresponding to $\ell^{\alpha}$ was presented in Atia--Littlejohn--Stewart \cite{Atia-Littlejohn-Stewart}.
The functions $y(x) = L_n^\alpha(x)$ satisfy the eigenvalue equation
\[
\ell^{\alpha}[y]=(n-1)y\quad(0<x<\infty).
\]

Probably the most comprehensive study of three types of exceptional Laguerre polynomials is contained in Liaw--Littlejohn--Milson--Stewart \cite{LLMS}. We refer the reader to Dur\'an \cite{DuranJAT2014} for an interesting relation with so-called exceptional Meixner polynomials.

Literature reveals several representations of the $X_1$-Laguerre polynomial sequence:

For example, in \cite{KMUG} as well as Ho--Sasaki \cite{HoSasakiZeros2012} the first couple of polynomials are listed
\begin{align}
\label{eq:list}
L_1^\alpha (x) &= x+\alpha +1,\\
\label{eq:list05}
L_2^\alpha (x) &= x^2 -\alpha(\alpha+2),\\
\label{eq:list1}
L_3^\alpha (x) &= \frac 1 2 \left[ -x^3 + (\alpha+3)x^2 + \alpha (\alpha+3) x - \alpha (\alpha^2+4\alpha+ 3)\right], \\
\notag&\vdots
\end{align}

Further, these polynomials can be expressed in terms of the classical Laguerre polynomials $\{p_n\}_{n\in\N_0}$ (here $p_{-1}\equiv 0$), see e.g.~\cite[Equation (3.2)]{LLMS}:
\[
L_n^\alpha(x) = -(x+\alpha+1) p_{n-1}^{\alpha-1}(x)+(x+\alpha)p_{n-2}^{\alpha}(x)
\qquad(n\in\N).
\]
Dur\'an's work \cite{DuranJAT2014} contains representations of general exceptional orthogonal polynomials using determinants of classical polynomials.

In \cite{KMUG} a three-term recurrence for the exceptional $X_1$-Laguerre polynomials was found
\begin{align*}
0
=&\quad\,
(n+1)[(x+\alpha)^2(n+\alpha)-\alpha]L_{n+2}^\alpha(x)\\
&+
(n+\alpha)[n+\alpha)^2(x-2n-\alpha-1)+2\alpha]L^\alpha_{n+1}(x)\\
&+
(n+\alpha-1)[(x+\alpha)^2 (n+\alpha+1) -\alpha]L_n^\alpha(x).
\end{align*}

We are most interested in the way the $X_1$-Laguerre polynomial sequence was introduced in the work of Gom\'ez-Ullate--Kamran--Milson \cite{KMUG6}. Namely, the sequence of polynomials $\{v_i(x)\}_{i=1}^\infty$ where
\begin{align}
\label{eq;vs}
v_1(x) = x+\alpha+1\quad\text{and}\quad v_i(x) = (x+\alpha)^i\text{ for }i \ge 2,
\end{align}
span the exceptional $X_1$-Laguerre polynomial flag. Via the Gram--Schmidt process this sequence produces the sequence $\{L_n^\alpha\}_{n=1}^\infty$.

We note that the classical orthogonal polynomials are obtained from the sequence $\{1, x, x^2, \hdots\}$ also by applying Gram--Schmidt. This perspective is used to find the classical moment representation, see e.g.~\cite[Chapter I.3]{Chihara}. An adaption of these ideas leads us to deduce our representations.

The proof of our recursion formula for the moments relies on an application of so-called symmetry factors for differential equations (see e.g.~Cole \cite[p.~66]{Cole}) which was further developed from second to higher order differential equations by Littlejohn, see e.g.~\cite{Littlejohn1983, Littlejohn1984}. In essence, every second order differential equation is symmetrizable, and the corresponding symmetry equation is solved by the weight (of orthogonality). In Krall's \cite{Krall1938} well-known classification theorem (see e.g.~Krall--Littlejohn \cite[Theorem 1(ii)]{KrallLittlejohn}) this theory was used to derive a moment equation for the classical moments. Here we carry out those techniques to obtain a recursive definition for the \emph{(adjusted) moments}
\begin{align*}
\wt\mu_n = \wt\mu_n^\alpha := \int_0^\infty (x+\alpha)^n W^\alpha(x) dx.
\end{align*}
A formula for the \emph{(exceptional) moments}
\begin{align*}
\mu_n =\mu_n^\alpha:= \int_0^\infty x^n W^\alpha(x) dx
\end{align*}
is obtained.

\begin{rem*}
We warn the reader that we often drop the subscript or superscript $\alpha$ to simplify notation.
\end{rem*}

\subsection{Plan of the paper}
In Section \ref{s-ExcCond} we present a preliminary result: We characterize the subspace spanned by the first $m$ exceptional $X_1$-Laguerre polynomials (Lemma \ref{lemma}), which yields what we call the exceptional condition satisfied by all exceptional $X_1$-Laguerre polynomials (Lemma \ref{lemma2}).

A first representation (Theorem \ref{t-FirstRep}) of the exceptional $X_1$-Laguerre polynomials in terms of the exceptional moments is found in Section \ref{s-FirstRep}. The expression is rather cumbersome.

In Section \ref{s-recursion} we focus on recursion formulae for the adjusted (Theorem \ref{t-AdMom}) and the exceptional moments (Theorem \ref{t-ExMom}). The adjusted moments are also expressed explicitly in two different ways (Corollary \ref{corollary} and Theorem \ref{t-generating}). One of the proofs involves a generating function. Throughout Subection \ref{ss-AdMom} we notice that the first two moments are different in nature than the others.

A more elegant representation (Theorem \ref{t-AltRep}) of the exceptional $X_1$-Laguerre polynomials in terms of the adjusted moments is the topic of Section \ref{s-AltRep}. In Remark \ref{r-normalization} we determine the normalization constant so as to yield precisely those polynomials in the literature \cite{KMUG} and \cite{LLMS}. At the very end, we verify the representation formula for $n=1$ and $n=2$.

The author considers Theorems \ref{t-generating} and \ref{t-AltRep} to be the main results of this paper.

\section{The exceptional condition and the polynomials $L_n^\alpha$}\label{s-ExcCond}
We characterize the span of the first $m$ exceptional $X_1$-Laguerre polynomials as those polynomials of degree less than or equal to $m$ for which the \emph{exceptional condition}
\begin{equation} \label{eq:span}
  p^\prime (-\alpha) - p(-\alpha) = 0
\end{equation}
holds. This fact is well-known to specialists, but it is often written slightly different. We include the proof for the convenience of the reader. We mention aside that it is not hard to see that the polynomials given by \eqref{eq:list} through \eqref{eq:list1}, of course, satisfy the exceptional condition.

In order to formulate precisely and prove the above characterization result, we let $\cP_m$ denote the set of polynomials of $\deg p \le m$ and define the span of the first $m$ exceptional $X_1$-Laguerre polynomials
\[
\cL_m :=\spa\{L_{n}^\alpha : n = 1, \hdots, m\},
\]
and
\[
\cM_m := \{p \in \cP_m : p\text{ satisfies }\eqref{eq:span}\}.
\]

\begin{lem}
\label{lemma}
The sets $ \cL_m = \cM_m $ for all $m\in \N$.
\end{lem}

\begin{proof} 
To show $\cL_m \subseteq \cM_m$ take $p\in \cL_m$.  Then $\ell^\alpha[p]\in \cL_m$, because $\cL_m$ is a subspace of $L^2(\mu^\alpha)$ which is invariant under $\ell^\alpha$. In particular $\ell^\alpha[p](x)$ is a polynomial in $x$. In virtue of the definition of $\ell^\alpha$, some cancellation has to take place in the exceptional term $$\left(\frac{x-\alpha}{x+\alpha}\right) \left[(x+\alpha+1)p'-p\right].$$ So the term in square brackets must equal zero when evaluated at $x=-\alpha$.
This yields precisely the exceptional condition \eqref{eq:span}. And so we have $\cL_m \subseteq \cM_m$.

The opposite containment follows from a dimension argument. Namely, we have $\dim\cP_m = m+1$. This implies $ \dim\cM_m = m$, due to the imposition of the one restriction $\eqref{eq:span}$.  We also have $\dim \cL_m = m$, because the set is spanned by $m$ polynomials with mutually different orders (implying their linear independence). Finally, we recall that $\cL_m$ and $\cM_m$ are both subspaces of $\cP_m$, and $\cL_m \subseteq \cM_m$.
  \end{proof}

The exceptional condition \eqref{eq:span} yields a condition on the coefficients of the polynomials $L_n^\alpha$. To see this, we write
\[
L_{n}^{\alpha}(x) = \sum_{k = 0}^{n} c_{nk} x^k 
\]
and compute $$\left(L_{n}^{\alpha}\right)'(x) = \sum_{k = 1}^{n} k c_{nk} x^{k-1}.$$ Since $L_n^\alpha\in \cL_n$, condition \eqref{eq:span} applies and we have
$
0 = \left(L_{n}^{\alpha}\right)'(-\alpha) - L_{n}^{\alpha}(-\alpha).
$ 
We conclude:
\begin{lem}
\label{lemma2}
The coefficients of the $X_1$-exceptional Laguerre polynomials $L_{n}^{\alpha}(x) = \sum_{k = 0}^{n} c_{nk} x^k $ obey
\begin{equation} \label{eq:main}
  -c_{n0} + \sum_{k=1}^n c_{nk} \left[ k(-\alpha)^{k-1} - (-\alpha)^k \right] = 0.
\end{equation}
\end{lem}

\section{First Representation of $L_{n}^{\alpha}$}\label{s-FirstRep}
Fix $n\in \N$. Recall that $L_{n}^{\alpha} = \sum_{k=0}^n c_{nk} x^k$. We determine $c_{nk}$ for $k=0,1,\dots,n$ by means of a linear system of $n+1$ equations
  $Ac = b$, where
  $$
  c := \left[ \begin{array}{c}
c_{n0} \\ c_{n1} \\ \vdots \\ c_{nn} \end{array} \right]\in \R^{n+1}
\qquad\text{and}\qquad  
b := \left[ \begin{array}{c}
0 \\  \vdots \\ 0 \\ K_n \end{array} \right]\in \R^{n+1}
$$
with normalizing constant $K_n $ and $A$ is given below.
In this section, we do not fix $K_n$, but merely assume $K_n\neq 0$.
In Remark \ref{r-normalization} below we determine the normalization constant such that we obtain exactly the polynomials described in \cite{KMUG} and \cite{LLMS}.

\begin{theo}
\label{t-FirstRep}
The exceptional $X_1$-Laguerre polynomials admit the representation
\[
L_{n}^\alpha(x)
= \frac{1}{\det  \, A} \sum_{k=0}^n (\det  \, A_k) x^k
 = \frac{K_n}{\det  \, A} \cdot 
\left| \begin{array}{c} (\text{First }n\text{ rows of }A) \\ 1 \quad x \quad x^2 \quad \hdots \quad x^n \\ \end{array} \right|
\qquad(n\in\N),
\]
where
\[
A= \left[ \begin{array}{cccc}
-1 & 1(-\alpha)^0 - (-\alpha)^1  & \hdots &  n(-\alpha)^{n-1} - (-\alpha)^n \\
&&&\\
\mu_1 + (\alpha + 1)\mu_0  & \mu_2 + (\alpha + 1)\mu_1  & \hdots &  \mu_{n+1} + (\alpha + 1)\mu_n \\
&&&\\
\displaystyle{\sum_{m=0}^2} \binom{2}{m} \, \mu_m \; \alpha^{2-m}  &\displaystyle{ \sum_{m=0}^2} \binom{2}{m} \, \mu_{m+1} \; \alpha^{2-m}  & \hdots &  \displaystyle{\sum_{m=0}^2} \binom{2}{m} \, \mu_{m+n} \; \alpha^{2-m} \\
&&&\\
\vdots & \vdots && \vdots \\
&&&\\
\displaystyle{ \sum_{m=0}^n} \binom{n}{m} \, \mu_m \; \alpha^{n-m}  &\displaystyle{ \sum_{m=0}^n }\binom{n}{m} \, \mu_{m+1} \; \alpha^{n-m}  & \hdots & \displaystyle{ \sum_{m=0}^n} \binom{n}{m} \, \mu_{m+n} \; \alpha^{n-m} \\ \end{array} \right] ,
 \]
 the exceptional moments are given by $\mu_k =\mu_k^\alpha = \int_0^\infty x^k W^\alpha(x) dx$,
 and where the matrix $A_k$ is obtained from $A$ by replacing the $(k+1)$-st column with the vector $b$; as is done in Cramer's rule.
(In Section \ref{s-recursion} below we find a recursion formula for the moments $\mu_k$.)
\end{theo}

\begin{rem*}
There is, indeed, no polynomial of order zero.
\end{rem*}

\begin{rem*}
The matrix $A$ is invertible, since the exceptional $X_1$-Laguerre polynomials are determined uniquely by exactly those conditions. Indeed, Lemma \ref{lemma} ensures that the polynomial $\sum_{k = 0}^{n} c_{nk} x^k$ belongs to the vector space $\cL_n$, and the other conditions given by the rows of the linear system simply require that it is orthogonal to the subspace $\cL_{n-1}$. Since $\dim \left(\cL_n \setminus \cL_{n-1}\right) =1$, the polynomial $L_n^\alpha$ is uniquely (up to choosing the normalizing constant $K_n\neq 0$) defined by these conditions.
\end{rem*}

\begin{proof}
One condition on the coefficients $c_{nk}$ was given by \eqref{eq:main}, which we rewrite in terms of the dot product
\[
\left[ \begin{array}{ccccc}
-1 \,\,\,&\,\,\, 1(-\alpha)^0 - (-\alpha)^1 \,\,\,&\,\,\, 2(-\alpha)^1 - (-\alpha)^2 \,\,\,&\,\,\, \hdots \,\,\,&\,\,\, n(-\alpha)^{n-1} - (-\alpha)^n \end{array} \right] \cdot 
c=0.
\]
We use the row vector in square brackets as the first row of the matrix $A$.

The other $n$ rows are obtained from orthogonality conditions: Recall that the polynomials $L_n^\alpha$ are obtained from the sequence $v_i$ given by \eqref{eq;vs} via the Gram--Schmidt process. In particular, we have
\[
\left< L_{n}^\alpha \,\,,\, v_k \right>_{W^\alpha} = K_n\delta_{nk}
\quad\text{for}\quad k = 1, \hdots , n,
\]
where $\delta_{nk}$ is the Kronecker delta.
We now find explicit expressions for these conditions for the different values of $k = 1, \hdots, n$.

For $k = 1$, we have
\begin{align*}
K_n\delta_{n1} &= \left< L_{n}^\alpha \,\,,\, v_1 \right>_{W^\alpha} \\
&=
 \left< \sum_{k=0}^n c_{nk}x^k \,\,,\; x+\alpha+1 \right>_{W^\alpha}\\ 
&=
\int_0^\infty \left( \sum_{k=0}^n c_{nk}x^{k+1} + (\alpha + 1)\sum_{k=0}^n c_{nk}x^k \right) \, W^\alpha(x) \, dx \\
&=
 \sum_{k=0}^n c_{nk} (\mu_{k+1} + (\alpha + 1)\mu_k) .
\end{align*}

The second row of the matrix $A$ is determined by the first factor of the dot product
\[
\left[ \begin{array}{cccc}
\mu_1 + (\alpha + 1)\mu_0 \,\,\,\,&\,\,\,\, \mu_2 + (\alpha + 1)\mu_1 \,\,\,\,&\,\,\,\, \hdots \,\,\,\,&\,\,\,\, \mu_{n+1} + (\alpha + 1)\mu_n \end{array} \right] 
\cdot 
c =K_n\delta_{n1} .
\]

When $n\ge 2$ we consider $2\le s\le n$ and use the binomial theorem,
\[
v_s = (x + \alpha)^s = \sum_{m=0}^s \binom{s}{m} \, x^m \, \alpha^{s-m}.
\]
We compute
\begin{align*}
K_n\delta_{ns} 
&=
\left< L_{n}^\alpha \,\,,\, v_s \right>_{W^\alpha}\\
&=\int_0^\infty \left(\sum_{k=0}^n c_{nk} \, x^k\right)\left(\sum_{m=0}^s \binom{s}{m} \, x^m \, \alpha^{s-m}\right) \, W^\alpha(x) \, dx\\
&= 
\int_0^\infty \sum_{k=0}^n \, c_{nk} \, \left(\sum_{m=0}^s \binom{s}{m} \, x^{m+k} \, \alpha^{s-m}\right) \, W^\alpha(x) \, dx
\\
&= \sum_{k=0}^n \, c_{nk} \, \left(\sum_{m=0}^s \binom{s}{m} \, \int_0^\infty x^{m+k} \, \alpha^{s-m} \, W^\alpha(x) \, dx\right)\\
&  = \sum_{k=0}^n \, c_{nk} \, \left(\sum_{m=0}^s \binom{s}{m} \, \mu_{m+k} \; \alpha^{s-m} \right).
\end{align*}

Rewriting the summation as a dot product as before, we find that the $(s+1)$-st row of the matrix $A$ is determined by
\[
\left[ \begin{array}{cccc}
\displaystyle{\sum_{m=0}^s} \binom{s}{m} \, \mu_m \; \alpha^{s-m} & \displaystyle{\sum_{m=0}^s} \binom{s}{m} \, \mu_{m+1} \; \alpha^{s-m} & \hdots & \displaystyle{\sum_{m=0}^s} \binom{s}{m} \, \mu_{m+n} \; \alpha^{s-m} \end{array} \right]
\cdot 
c=K_n\delta_{ns} .
\]

We solve the linear system $Ac = b$ using Cramer's rule
\[
c_{nk}
= 
(\det A_k) /(\det A)
\qquad(\text{for }k=0,1, \hdots, n).
\]

It remains to build the polynomial from those coefficients $c_{nk}$, $k=0,1, \hdots, n$.
With the definition of the vector $b$ we easily obtain
\begin{equation*}
\left| \begin{array}{c}
    (\text{First } n \text{ rows of }A)  \\
  0 \,\, \hdots\,\,  0  \quad x^k\quad  0\,\, \hdots \,\, 0
  \end{array} \right|
  =  \frac{(\det  \, A_k)}{K_n} \, x^k .
\end{equation*}
In the above determinant the entry $x^k$ is in the $(k+1)$-st column of the matrix.

So the desired formula
\begin{align*}
\frac{1}{\det  \, A} \sum_{k=0}^n (\det  \, A_k) x^k
 = \frac{K_n}{\det  \, A} \cdot 
\left| \begin{array}{c} (\text{First }n\text{ rows of }A) \\ 1 \quad x \quad x^2 \quad \hdots \quad x^n \\ \end{array} \right|
\end{align*}
follows from expansion of the matrix by minors along the last row.
\end{proof}

\section{Moment Formulas}
\label{s-recursion}
Because our representation for $L_{n}^\alpha(x)$ effectively depends on the moments $\mu_k$, we now develop a recursive formula for the moments.  We do this in an indirect fashion.
Namely, rather than working with the moments $\mu_k$, we initially develop a recursion formula for the adjusted moments $\wt\mu_k $. We then apply the binomial formula to derive the expression for $\mu_k$. We choose this indirect approach because it became evident during the course of our investigation that the moment calculations for the adjusted moments are more concise. 
(It is possible to do the computations for the exceptional moments $\mu_k$ directly, but such a direct computation turns out to be rather messy.)

\subsection{Adjusted Moments}\label{ss-AdMom}

\begin{theo}
\label{t-AdMom}
The adjusted moments $\wt\mu_k= \int_0^\infty (x+\alpha)^k \, W^\alpha (x) \, dx$
\begin{itemize}
\item[(a)] satisfy the recursion formula
\begin{equation*}
  \widetilde{\mu}_{k+2} = (2\alpha + k) \widetilde{\mu}_{k+1} + \alpha(1-k) \widetilde{\mu}_k
  \qquad(k\in \N_0),
\end{equation*}
\item[(b)] and we can start the recursion with
\begin{align}
\label{eq:WTmu1}
\widetilde{\mu}_1^\alpha 
&= 
 e^\alpha \alpha^\alpha \Gamma(1+\alpha) \Gamma(-\alpha,\alpha)
  ,\text{ and}\\
\notag
\wt\mu_0^\alpha 
&
= \Gamma(\alpha) - 2e^\alpha \alpha^\alpha \Gamma(\alpha+1) \Gamma(-\alpha,\alpha),
\end{align}
where the Gamma function is given by $\Gamma(x):=\int_0^\infty t^{x-1} e^{-t} dt$ and the incomplete Gamma function by $\Gamma(a,x) := \int_x^\infty t^{a-1} e^{-t} dt$ for $x>0$.
\end{itemize}
\end{theo}

We simplify notation by writing $W$ for $W^\alpha(x)$.

\begin{proof}
We begin the proof of part (a) by observing two facts.

First, for functions $f,g$ smooth on $[0,\infty)$ the moment functionals satisfy:
\[
\left\langle W', f \right\rangle = - \left\langle W, f' \right\rangle
\quad
\text{and}\quad
\left\langle g W, f \right\rangle = \left\langle W, fg \right\rangle
\]
(here $\langle\fdot, \fdot\rangle$ denotes the inner product with respect to Lebesgue measure on $[0,\infty)$).

Second, we learn from \cite{Littlejohn1984} that for a linear operator of the form $$\ell[y] = a_2y'' + a_1y' + a_0y,$$ the related symmetry equation is given by $$a_2y' + (a'_2 - a_1)y = 0,$$ and that it is solved by the weight function (with respect to which the eigen-polynomials are orthogonal).

That is
\[
a_2W' + (a_2' - a_1)W=0.
\]
And together with
\begin{align*}
\left\langle a_2W', \; (x+\alpha)^k \right\rangle
&=
\left\langle W', \; a_2(x+\alpha)^k \right\rangle
=
-\left\langle W, \; (a_2(x+\alpha)^k)' \right\rangle\\
&=
-k\left\langle W, \; a_2(x+\alpha)^{k-1} \right\rangle - \left\langle W, \; a_2'(x+\alpha)^k \right\rangle
\end{align*}
we find for $k\in \N$:
\begin{align}
\notag
0 
&= \left\langle a_2W' + (a_2' - a_1)W, \; (x+\alpha)^k \right\rangle\\
\notag
&=
\left\langle a_2W', \; (x+\alpha)^k \right\rangle +  \left\langle a_2'W, \; (x+\alpha)^k \right\rangle - \left\langle a_1W, \; (x+\alpha)^k \right\rangle\\
\label{eq:innerprod}
&=
-k\left\langle W, \; a_2(x+\alpha)^{k-1} \right\rangle - \left\langle a_1W, \; (x+\alpha)^k \right\rangle.
\end{align}

In our case, with the exceptional $X_1$-Laguerre expression \eqref{eq:diff} the coefficients are $$a_2 = -x\quad\text{and}\quad a_1 = x - \alpha - 1 + \frac{2x}{x+\alpha}\,,
$$
or equivalently,
$$
a_2 = -(x+\alpha) + \alpha\quad\text{and}\quad
a_1 = (x + \alpha) - 2\alpha + 1 - \frac{2\alpha}{x+\alpha}\,.
$$

We substitute these into \eqref{eq:innerprod}, and collect the coefficients of terms of the form $\widetilde{\mu}_k = \left\langle W, (x+\alpha)^k \right\rangle$ for $k=k-1, k, k+1$ to obtain:
\begin{align*}
0 
&=
 -k\left\langle W, \; (-(x+\alpha) + \alpha)(x+\alpha)^{k-1} \right\rangle \\
 &\quad
\,- \left\langle \left( (x + \alpha) - 2\alpha + 1 - \frac{2\alpha}{x+\alpha} \right) \! W, \; (x+\alpha)^k \right\rangle\\
&= 
[-k\alpha+2\alpha]\wt\mu_{k-1}
+
[k+2\alpha]\wt\mu_{k}
-
\wt\mu_{k+1}.
\end{align*}

We solve for $\wt\mu_{k+1}$
\[
\widetilde{\mu}_{k+1} = (2\alpha + k - 1) \widetilde{\mu}_k + \alpha(2-k) \widetilde{\mu}_{k-1}.
\]
Finally, shifting the index up by one, we see part (a) of the theorem.

We proceed to prove part (b). We recall the definition of $\Gamma(x)$ and $\Gamma(a,x)$, and let 
\[ E_a(x) = \int_1^\infty e^{-x t} t^{-a} dt ,\quad x>0 \]
denote the exponential integral function.
The two classes of functions are
related by
\[ E_a(x) = x^{a-1} \Gamma(1-a,x).\]
We also have the identity
\begin{equation}
  \label{eq:Earecursion}
  (a-1) E_a(x) = e^{-x} - x E_{a-1}(x).
\end{equation}

Our first claim is
\begin{align}
  \label{eq:claim1}
  \int_0^\infty \frac{e^{-x} x^\beta}{(x+\alpha)} dx 
  &= e^\alpha\, E_{1+\beta}(\alpha)\, \Gamma(1+\beta),\quad
  \alpha>0,\beta >-1.
\end{align}
The claim is established by the following chain of manipulations:
\begin{align*}
  \int_0^\infty \frac{e^{-x} x^\beta}{(x+\alpha)} dx &
  =
  \int_1^\infty e^{\alpha(1- t)} (\alpha (t-1))^\beta t^{-1} dt 
  \\
  &
  =  \int_\alpha^\infty e^{\alpha-s} (s-\alpha)^\beta s^{-1} ds \\
  &= e^\alpha \int_\alpha^\infty \int_1^\infty e^{-st} (s-\alpha)^\beta dt ds
  \\
  &
  = e^\alpha \int_1^\infty \int_0^\infty e^{-t\alpha-u} (u/t)^\beta t^{-1} du dt\\
  &= e^\alpha \int_1^\infty e^{-t\alpha} t^{-\beta-1} dt \int_0^\infty
  e^{-u} u^\beta du.
\end{align*}

As an immediate consequence, we obtain the following expressions for
the first adjusted moment:
\begin{align*}
  \wt\mu_1 &= e^\alpha E_{1+\alpha}(\alpha) \Gamma(1+\alpha)
  \\
  &
  = \Gamma(\alpha)(1 - \alpha e^\alpha E_\alpha(\alpha))
  \\
  &
  = e^\alpha \alpha^\alpha \Gamma(1+\alpha) \Gamma(-\alpha,\alpha).
\end{align*}

Further, notice that
\begin{align}
\label{WTmu2}
\wt\mu_2 = \int_0^\infty (x+\alpha)^2 \frac{x^{\alpha}e^{-x}}{(x+\alpha)^2}dx  = \int_0^\infty x^\alpha e^{-x} dx= \Gamma(\alpha+1).
\end{align}

With part (a) for $k=0$, we obtain $\Gamma(\alpha+1) = 2\alpha\wt\mu_1 +\alpha\wt\mu_0$ or equivalently
\[
\wt\mu_0
=
\Gamma(\alpha+1)/\alpha - 2\wt\mu_1 = \Gamma(\alpha) - 2e^\alpha \alpha^\alpha \Gamma(\alpha+1) \Gamma(-\alpha,\alpha).
\]
The theorem is proved.
\end{proof}

Equation \eqref{WTmu2} is of interest by itself.

\begin{rem*}
R.~Milson contributed the evaluation of the moment $\wt\mu_1^\alpha$.
\end{rem*}

In the remainder of this subsection, we express the $(k+2)$-nd adjusted moment $\wt\mu_{k+2}$ ``directly" in two ways. 
Define the matrix
\[
B_n:=
\left[ \begin{array}{cc} 2\alpha+n &\alpha(1-n) \\ 1&0 \end{array} \right].
\]

\begin{cor}
\label{corollary}
For $k\ge2$
\[
\left[ \begin{array}{c} \wt\mu_{k+2} \\ \wt\mu_{k+1} \end{array} \right]
=
\Gamma(\alpha+1)\left(\prod_{n=2}^k B_n\right)
\left[ \begin{array}{c} 2\alpha+1 \\ 1 \end{array} \right].
\]
Further, we have $\wt\mu_2 = \Gamma(\alpha+1)$ and $\wt\mu_3 = (2\alpha+1)\Gamma(\alpha+1).$

(The moments $\wt\mu_0$ and $\wt\mu_1$ were given explicitly in Theorem \ref{t-AdMom}.)
\end{cor}

\begin{rem*}
The first two moments $\wt\mu_0$ and $\wt\mu_1$ occur somewhat disconnected from the other moments $\wt\mu_k$ for $k\ge 2$. The forward recursion can be started from $k=2$, and that the formulas for $\wt\mu_0$ and $\wt\mu_1$ contain the incomplete gamma function $\Gamma.$ This observation also reflects the fact that the exceptional $X_1$-Laguerre polynomials do not have the degree zero polynomial.
\end{rem*}

\begin{proof}
By the recursion formula in Theorem \ref{t-AdMom} we can set up a discrete dynamical system to express
\[
\left[ \begin{array}{c} \wt\mu_{k+2} \\ \wt\mu_{k+1} \end{array} \right]
=
B_k
\left[ \begin{array}{c} \wt\mu_{k+1} \\ \wt\mu_{k} \end{array} \right].
\]

Repeated application yields
\begin{align}
\label{e-prodA}
\left[ \begin{array}{c} \wt\mu_{k+2} \\ \wt\mu_{k+1} \end{array} \right]
=
\left(\prod_{n=2}^k B_n\right)
\left[ \begin{array}{c} \wt\mu_{3} \\ \wt\mu_{2} \end{array} \right].
\end{align}

We recall that $\wt\mu_2 = \Gamma(\alpha+1)$ by equation \eqref{WTmu2}. Further with the moment recursion formula in Theorem \ref{t-AdMom} for $k=1$ we have
\begin{align}
\label{wt2and3}
\wt\mu_3 = (2\alpha+1)\wt\mu_2 = (2\alpha+1)\Gamma(\alpha+1).
\end{align}

Substitution into the vector on the right hand side of \eqref{e-prodA} yields the corollary.
\end{proof}

We use the standard technique of generating functions (see e.g.~\cite{IsmailBOOK}) to find an explicit expression of the moment.

\begin{theo}
\label{t-generating}
The adjusted moments $\wt\mu_k= \int_0^\infty (x+\alpha)^k \, W^\alpha (x) \, dx$ are given in hypergeometric notation by
\[
\wt\mu_{k+2}^\alpha
 = (-1)^k \Gamma(\alpha+1)
 (-\alpha-k)_k  \,\,
 {}_1F_1(-k,-\alpha-k;\alpha)
   \qquad(k\in \N_0),
\]
where we used the Pochhammer symbol
$$
(x)_n
: =
\left\{
\begin{array}{ll}
1 & \text{for }n=0\\
x(x+1)\cdot\hdots\cdot(x+n-1)
&
\text{for }n>0.
\end{array}
\right.
$$

(Again, the moments $\wt\mu_0^\alpha$ and $\wt\mu_1^\alpha$ cannot be obtained in this fashion, but their values are given in Theorem \ref{t-AdMom}.)
\end{theo}

We check that indeed $\wt\mu_2^\alpha = \Gamma(\alpha+1)$ and $\wt\mu_3^\alpha = (2\alpha+1) \Gamma(\alpha+1)$.
The hypergeometric lay may consult equation \eqref{e-nonhyper} below to find an alternative expression for the adjusted moments without hypergeometric notation.

\begin{proof}
We begin by re-writing the recurrence relation from Theorem \ref{t-AdMom}:
\begin{equation}
\label{e-AdMom}
  \widetilde{\mu}_{k+2} = (2\alpha + k) \widetilde{\mu}_{k+1} + \alpha(1-k) \widetilde{\mu}_k
  \qquad(k\in \N_0)
\end{equation}
to obtain
\begin{align}
\label{e-newRela}
(k+1)\nu_{k+1}
=
(2\alpha +k+1) \nu_k
-
\alpha \nu_{k-1}
  \qquad(k\in \N),
\end{align}
where
\begin{align}
\label{e-munu}
\wt\mu_{k+2} = k!\nu_k
  \qquad(k \in\N_0).
\end{align}

To see this we first replace $\tau_k:=\wt\mu_{k+2}$ in \eqref{e-AdMom} to see
\begin{equation*}
  \tau_{k} = (2\alpha + k) \tau_{k-1} + \alpha(1-k) \tau_{k-2}
  \qquad(k = 2, 3, 4, \hdots).
\end{equation*}
Shifting the index $k\mapsto k+1$ yields
\begin{equation*}
  \tau_{k+1} = (2\alpha + k+1) \tau_{k} - \alpha k \tau_{k-1}
  \qquad(k \in\N).
\end{equation*}
We define the sequence $\{\nu_k\}_{k\in \N}$ by $\tau_k = k! \nu_k$ and so 
\[
(k+1)! \nu_{k+1} = (2\alpha + k+1) k! \nu_{k} - \alpha k! \nu_{k-1} \qquad(k \in\N).
\]
We divide the equation by $k!$ to arrive at the desired relation \eqref{e-newRela}.

Our next goal is to write a related first order differential equation.
In order to achieve this, we multiply relation \eqref{e-newRela} by the factor $t^k$ and sum up for $k\in \N$:
\begin{align}
\label{e-sum}
\sum_{k=1}^\infty
(k+1)\nu_{k+1} t^k
=
\sum_{k=1}^\infty
(2\alpha +k+1) \nu_k t^k
-
\sum_{k=1}^\infty
\alpha \nu_{k-1} t^k
.\end{align}

In what follows, $t$ is treated as the independent variable. 
We define the generating function
\begin{align}
\label{e-GBC}
G(t) := 
\sum_{k=0}^\infty
\nu_k t^k.
\end{align}
Note that the moments can be expressed by
\begin{align}
\label{e-MomentG}
\wt\mu^\alpha_{k+2}
=
k! \nu_k
=
G^{(k)}(0).
\end{align}
In order to write \eqref{e-sum} using $G(t),$ we substitute
\begin{align*}
 \sum_{k=1}^\infty
(k+1)\nu_{k+1} t^{k} 
&
= G'(t) -\nu_1
\end{align*}
on the left hand side, and
\begin{align*}
(2\alpha+1)
\sum_{k=1}^\infty
\nu_{k} t^{k} 
+\sum_{k=1}^\infty
k \nu_k t^k
&
=
 (2\alpha+1)[G(t) - \nu_0]
 +tG'(t), \text{ as well as}
\\
-
\sum_{k=1}^\infty
\alpha \nu_{k-1} t^k
=
-\alpha t
\sum_{k=0}^\infty \nu_{k} t^k
&
=
-\alpha t G(t).
\end{align*}
Apriori, we expect a first order \emph{inhomogeneous} differential equation. However, when we collect the terms without the generating function (i.e.~the inhomogeneity) we see
\[
\nu_1 - (2\alpha+1) \nu_0 = 
\wt\mu_3 - (2\alpha+1) \wt\mu_2 = 
0
\]
by Equation \eqref{wt2and3}. We obtain the first order \emph{homogeneous} differential equation
\begin{align}
\label{e-firstorderDE}
(1-t)G'(t)+(\alpha t- 2\alpha -1)G(t)
=
0.
\end{align}
The boundary condition
\begin{align}
\label{e-BC}
G(0) = \nu_0 = \wt\mu^\alpha_2 = \Gamma(\alpha+1)
\end{align}
follows from \eqref{e-GBC} for $t=0$.

It is not hard to see that (near $t=0$) the generating function
\[
G(t) = \Gamma(\alpha+1) e^{\alpha(t-1)} (1-t)^{-(\alpha+1)}
\]
solves this boundary value problem.

With the Leibniz rule and elementary simplifications we derive
\begin{align}
\label{e-nonhyper}
\wt\mu_{k+2}^\alpha
 = 
  \Gamma(\alpha+1)
\sum_{m=0}^{k}
(-1)^{k-m} \binom{k}{m}\alpha^m
(-\alpha-k+m)_{k-m}
   \qquad(k\in \N_0).
\end{align}
The theorem now follows from the standard hypergeometric notation by using the identity
\[
(-\alpha-k+m)_{k-m}
=\frac{
(-\alpha-k)_k}{(-\alpha-k)_m}\,.
\]
(We note that the numerator in the latter expression is independent of the summation variable $m$.)
\end{proof}

\begin{rem*}
This method along with a first version of the first order differential equation was mentioned to us by M.E.H.~Ismail.
\end{rem*}

\subsection{Exceptional Moments}

Starting from the recursion formula in Theorem \ref{t-AdMom} for the adjusted moments, we now derive a recursion formula for the exceptional moments by using the binomial formula.

\begin{theo}
\label{t-ExMom}
The moments $\mu_k = \int_0^\infty x^k \, W^\alpha (x) \, dx$
\begin{itemize}
\item[(a)]
satisfy the recursion formula
$$
\mu_{k+2} = \sum_{m=0}^k \left[ (2\alpha + k)\binom{k+1}{m} - \alpha \binom{k+2}{m} + (1-k) \binom{k}{m} \right] \alpha^{k+1-m} \, \mu_m + \, (1 - \alpha) k\mu_{k+1}
$$
for $k\in \N_0$.
Specifically, when $k = 0$, we find that $\mu_2 = \alpha(\alpha + 1)\mu_0$.
\item[(b)] with
\begin{align*}
\mu_0^\alpha 
&=
\Gamma(\alpha) - 2e^\alpha \alpha^\alpha \Gamma(\alpha+1) \Gamma(-\alpha,\alpha)
\\
\mu_1^\alpha 
&= 
-\Gamma(\alpha+1)
+[2\alpha+1]e^\alpha \alpha^\alpha \Gamma(\alpha+1)\Gamma(-\alpha, \alpha).
\end{align*}
\end{itemize}
\end{theo}

\begin{proof}
To prove part (b) note that
\[
\mu_0^\alpha = \int_0^\infty x^0 W^\alpha(x)dx  = \int_0^\infty (x+\alpha)^0 W^\alpha(x)dx = \widetilde{\mu}_0^\alpha 
\]
and the first statement follows. To see the second statement
\[
\wt\mu_1^\alpha = \int_0^\infty (x+\alpha) W^\alpha(x)dx = \int_0^\infty x^1 W^\alpha(x)dx+ \alpha\int_0^\infty x^0 W^\alpha(x)dx=  \mu_1^\alpha +\alpha \mu_0^\alpha.
\]
So we have
\[
\mu_1^\alpha 
= \wt\mu_1^\alpha- \alpha \mu_0^\alpha.
\]
Substituting \eqref{eq:WTmu1} as well as the formula for $\mu_0^\alpha$ we obtain part (b).

Part (a) follows from lengthy but elementary computations, which we present for the convenience of the reader.
The binomial theorem yields
\begin{align}
\notag
\widetilde{\mu}_k 
&= 
\int_0^\infty \left(\sum_{m=0}^k \binom{k}{m} x^m \alpha^{k-m}\right) W^\alpha (x) \, dx\\
\notag
&=
\sum_{m=0}^k \binom{k}{m} \alpha^{k-m} \int_0^\infty x^m W^\alpha (x) \, dx\\
\label{eq:Bin}
\widetilde{\mu}_k 
&=
\sum_{m=0}^k \binom{k}{m} \, \mu_m \, \alpha^{k-m}.
\end{align}
By splitting off the last term see that
\[
\widetilde{\mu}_{k+2} = \sum_{m=0}^{k+2} \binom{k+2}{m} \, \mu_m \, \alpha^{k+2-m}
 = \sum_{m=0}^{k+1} \binom{k+2}{m} \, \mu_m \, \alpha^{k+2-m} +  \mu_{k+2},
 \]
that is,
\[
\mu_{k+2} = -\sum_{m=0}^{k+1} \binom{k+2}{m} \, \mu_m \, \alpha^{k+2-m} + \, \widetilde{\mu}_{k+2}.
\]

Replacing $\wt\mu_{k+2}$ by the recursion formula for the adjusted moments in Theorem \ref{t-AdMom} and using \eqref{eq:Bin} twice, we have 
  \begin{align}
  \notag
  \mu_{k+2} =
  & -\sum_{m=0}^{k+1} \binom{k+2}{m} \, \mu_m \, \alpha^{k+2-m}
  + (2\alpha+k) \wt\mu_{k+1}
  + \, \alpha(1-k) \wt\mu_k
  \\
  \label{eq;plugin}
  \mu_{k+2}   =
  & -\sum_{m=0}^{k+1} \binom{k+2}{m} \, \mu_m \, \alpha^{k+2-m}\\
  & + \, (2\alpha+k) \sum_{m=0}^{k+1} \binom{k+1}{m} \, \mu_m \, \alpha^{k+1-m} \\
  \label{eq;pluginA}
  &+ \, \alpha(1-k) \sum_{m=0}^k \binom{k}{m} \, \mu_m \, \alpha^{k-m}.
 \end{align}

Finally, we combine the sums. This is done by using
\begin{align*}
&
-\sum_{m=0}^{k+1} \binom{k+2}{m} \, \mu_m \, \alpha^{k+2-m}
+2\alpha \sum_{m=0}^{k+1} \binom{k+1}{m} \, \mu_m \, \alpha^{k+1-m}\\
=&
\alpha \left[-\sum_{m=0}^{k+1} \binom{k+2}{m} \, \mu_m \, \alpha^{k+1-m}
+ 2 \sum_{m=0}^{k+1} \binom{k+1}{m} \, \mu_m \, \alpha^{k+1-m}\right]\\
=&
\alpha \sum_{m=0}^{k+1} \left[ 2\binom{k+1}{m} - \binom{k+2}{m} \right] \mu_m \, \alpha^{k+1-m}\\
=&
-k \alpha \mu_{k+1} 
+\alpha \sum_{m=0}^{k} \left[ 2\binom{k+1}{m} - \binom{k+2}{m} \right] \mu_m \, \alpha^{k+1-m}
\end{align*}
--- where the term in square brackets for $m=k+1$ evaluated as follows
\[
\left[ 2\binom{k+1}{k+1} - \binom{k+2}{k+1} \right]
= 2 -(k+2)=-k
\]
--- as well as 
\begin{align*}
&k \sum_{m=0}^{k+1} \binom{k+1}{m} \, \mu_m \, \alpha^{k+1-m}+ \, \alpha(1-k) \sum_{m=0}^k \binom{k}{m} \, \mu_m \, \alpha^{k-m}\\
= &
 k\mu_{k+1}+
 l \sum_{m=0}^k \binom{k+1}{m} \, \mu_m \, \alpha^{k+1-m} + \, (1-k) \sum_{m=0}^k \binom{k}{m} \, \mu_m \, \alpha^{k+1-m}\\
=&
 k\mu_{k+1}
 +\sum_{m=0}^{k} \left[ k \binom{k+1}{m} +(1- k) \binom{k}{m}  \right] \mu_m \, \alpha^{k+1-m}.
\end{align*}

Substitution into  \eqref{eq;plugin} through \eqref{eq;pluginA} yields
  \begin{align}
  \notag
  \mu_{k+2} =
  &-k \alpha \mu_{k+1} 
+\alpha \sum_{m=0}^{k} \left[ 2\binom{k+1}{m} - \binom{k+2}{m} \right] \mu_m \, \alpha^{k+1-m}\\
&+
 k\mu_{k+1}
 +\sum_{m=0}^{k} \left[ k \binom{k+1}{m} +(1- l) \binom{k}{m}  \right] \mu_m \, \alpha^{k+1-m}
  \end{align}
and we can easily verify the coefficients in the theorem.
\end{proof}

\section{Alternative Representation of $L_{n}^{\alpha}$ with Adjusted Moments}\label{s-AltRep}
In Section \ref{s-recursion} we saw that the adjusted moments can be computed much more easily. Here we obtain the more elegant representation for the exceptional $X_1$-Laguerre polynomials.

We begin by expressing the $X_1$-Laguerre orthogonal polynomials in terms of powers of $(x+\alpha)$.
Consider
\[
L_{n}^{\alpha}(x) = \sum_{k=0}^n a_{nk} (x+\alpha)^k.
\]
As in Section \ref{s-FirstRep}, the idea is to express the coefficients $a_{nk}$, $k=0,1, \hdots, n$, via a system of $n+1$ linear equations $\wt Aa = b$, where
\[
a = \left[ \begin{array}{c} a_{n0} \\ a_{n1} \\ \vdots \\ a_{nn} \end{array} \right]\in \R^{n+1}
\qquad\text{and}\qquad
b = \left[ \begin{array}{c} 0 \\ \vdots \\ 0 \\ \wt K_n \end{array} \right]\in \R^{n+1},
\]
with $\wt K_n \neq 0$ and where $\wt A$ is given in the next theorem:

\begin{theo}
\label{t-AltRep}
The exceptional $X_1$-Laguerre polynomials admit the representation
\[
L_{n}^\alpha(x)
= \frac{1}{\det  \, \wt A} \sum_{k=0}^n \left(\det  \, \wt A_k\right) (x+\alpha)^k
 = \frac{\wt K_n}{\det  \, \wt A}  
\left| \begin{array}{c} \left(\text{First }n\text{ rows of the matrix }\wt A\right) \\ 1 \,\, (x+\alpha) \,\, (x+\alpha)^2 \,\, \hdots \,\, (x+\alpha)^n \\ \end{array} \right|
\]
for $n\in\N$, where
where
\[
\wt A= 
 \left[ \begin{array}{ccccc}
-1 & 1 & 0 &\hdots & 0 \\
\widetilde{\mu}_0 + \widetilde{\mu}_1 & 
\widetilde{\mu}_1 + \widetilde{\mu}_2 & \hdots &\hdots & \widetilde{\mu}_n + \widetilde{\mu}_{n+1} \\
\widetilde{\mu}_2 & \widetilde{\mu}_3 & \hdots &\hdots & \widetilde{\mu}_{n+2} \\
\vdots & \vdots &  &&  \vdots \\
\widetilde{\mu}_n & \widetilde{\mu}_{n+1} & \hdots &\hdots & \widetilde{\mu}_{2n} \\ \end{array} \right],
 \]
 the exceptional moments are given by $\wt\mu_k = \int_0^\infty (x+\alpha)^k W^\alpha(x) dx$,
 and where the matrix $\wt A_k$ is obtained from $\wt A$ by replacing the $(k+1)$-st column with the vector $b$; as is done in Cramer's rule.
 (In Section \ref{s-recursion} we found a recursion formula for the moments $\wt\mu_k$.)
\end{theo}

\begin{proof}
The idea of establishing matrix $\wt A$ is the same way as was for matrix $A$ in Theorem \ref{t-FirstRep}. Since the moments are adjusted to better assimilate $v_k(x) = (x+\alpha)^k$, $k\ge 2$, from equations \eqref{eq;vs}, the binomial formula is not required and $\wt A$ turns out simpler that $A$.

The first row of $\wt A$ also simplifies: As before, we use the exceptional condition \eqref{eq:span}.
We obtain $$\left(L_{n}^{\alpha}\right)'(x) = \sum_{k = 1}^{n} k a_{nk} (x+\alpha)^{k-1}$$ and so $$L_n^\alpha(-\alpha) = a_{n0}\quad \text{as well as}\quad\left(L_{n}^{\alpha}\right)'(-\alpha) = a_{n1}.$$ We have
\[
0= \left(L_{n}^{\alpha}\right)'(-\alpha) - L_{n}^{\alpha}(-\alpha) 
= a_{n1}-a_{n0} .
\]
The entries in the first row of $\wt A$ follow.

The other $n$ conditions are, again, obtained via orthogonality.
So we have
\[
\left< L_{n}^\alpha \,\,,\, v_k \right>_{W^\alpha} = \wt K_n\delta_{nk}
\quad (k = 1, \hdots \, , n).
\]

For $k = 1$, we observe as before
\begin{align*}
\wt K_n\delta_{n1} 
&= \left< L_{n}^\alpha \,\,,\, v_1 \right>_{W^\alpha} \\
&= \left< \sum_{k=0}^n a_{nk}(x+\alpha)^k \,\,,\; x+\alpha+1 \right>_{W^\alpha}\\
&= \int_0^\infty \left(\sum_{k=0}^n a_{nk}(x+\alpha)^k\right)(x+\alpha+1) \, W^\alpha(x) \, dx\\
 &
 =
  \sum_{k=0}^n a_{nk} (\widetilde{\mu}_k + \widetilde{\mu}_{k+1}).
\end{align*}
So the second row of the matrix $\wt A$ is the first vector in the dot product
\[
\left[ \begin{array}{cccc}
\widetilde{\mu}_0 + \widetilde{\mu}_1 \,\,\,\,&\,\,\,\, 
\widetilde{\mu}_1 + \widetilde{\mu}_2 \,\,\,\,&\,\,\,\, \hdots \,\,\,\,&\,\,\,\, \widetilde{\mu}_n + \widetilde{\mu}_{n+1} \end{array} \right] \cdot 
a
= \wt K_n\delta_{n1}.
\]

When $n\ge 2$ we consider $2\le s\le n$. Again we compute
\begin{align*}
\wt K_n\delta_{ns} 
&= 
\left< L_{n}^\alpha \,\,,\, v_s \right>_{W^\alpha} \\
&= \left< \sum_{k=0}^n a_{nk}(x+\alpha)^k \,\,,\; (x+\alpha)^s \right>_{W^\alpha} \\
&
  = \sum_{k=0}^n a_{nk} \; \widetilde{\mu}_{k+s}.
\end{align*}
The corresponding coefficients fill the $(s+1)$-st row of the matrix $\wt A$.

Again, the matrix $\wt A$ is invertible and we apply Cramer's rule
\[
a_{nk}
= 
\left(\det \wt A_k\right) /\left(\det \wt A\right)
\qquad(\text{for }k=0,1, \hdots, n).
\]
With the definition of the vector $b$:
\begin{equation*}
\left| \begin{array}{c}
    (\text{First } n \text{ rows of }A)  \\
  0 \,\, \hdots\,\,  0  \quad (x+\alpha)^k\quad  0\,\, \hdots \,\, 0
  \end{array} \right|
  =  \frac{\left(\det  \wt A_k\right)}{\wt K_n} \, (x+\alpha)^k ,
\end{equation*}
and the desired formula follows from expansion of the matrix by minors along the last row.
\end{proof}

\begin{rem}
\label{r-normalization}
In order to embed this representation into the literature, we relate to the normalization used in \cite{KMUG} and \cite{LLMS}: The choice
\[
\wt K_n
=
(-1)^n(\alpha+n) \Gamma(\alpha+n-1)
\]
yields the same normalization as in \cite{KMUG} and \cite{LLMS}.

Indeed, there the leading coefficient of $L_n^\alpha(x)$ is given by $(-1)^n/(n-1)!$ and with this normalization they obtained
\[
\|L_n^\alpha\|^2 = \frac{\Gamma(\alpha+n-1)(\alpha+n)}{(n-1)!}.
\]
And so by the Gram--Schmidt orthogonalization we have
\[
\wt K_n
=
\left\langle
L_n^\alpha(x), (x+\alpha)^n
\right\rangle
=
(n-1)!/(-1)^n
\|L_n^\alpha\|^2
=
(-1)^n(\alpha+n) \Gamma(\alpha+n-1).
\]
This relation was pointed out to us by R.~Milson.
\end{rem}

Finally, we verify for $n= 1$ and $n=2$ that the polynomials in Theorem \ref{t-AltRep} indeed agree with \eqref{eq:list} and \eqref{eq:list05}, respectively, as well as the normalization claimed in the latter remark.

\begin{ex*}
For $n=1$ compute
\[
\left| \begin{array}{cc}
-1 & 1\\
1&x+\alpha \end{array} \right| = -x-\alpha -1,
\]
which is a scalar multiple of \eqref{eq:list}. And, moreover, normalizing the formula in Theorem \ref{t-AltRep} according to Remark \ref{r-normalization} we obtain
\[
\frac{\wt K_n}{\det  \, \wt A}
\left| \begin{array}{cc}
-1 & 1\\
1&x+\alpha \end{array} \right|
= 
\frac{-(\alpha+1)\Gamma(\alpha)}{(-\alpha-1)\Gamma(\alpha)}(-x-\alpha -1)
=
-x-\alpha -1,
\]
the leading term of which is in agreement with the leading term in Remark \ref{r-normalization}, since  for $n=1$ we have $(-1)^n/(n-1)! =-1$.
\end{ex*}

\begin{ex*}
Take $n=2$. By co-factor expansion along the last row we evaluate
\begin{align*}
L_2^\alpha(x)
&
=
\left| \begin{array}{ccc}
-1 & 1 & 0 \\
\widetilde{\mu}_0 + \widetilde{\mu}_1 & 
\widetilde{\mu}_1 + \widetilde{\mu}_2 & \widetilde{\mu}_2 + \widetilde{\mu}_{3} \\
1&x+\alpha &(x+\alpha)^2 \end{array} \right|
\\&
=
\wt\mu_2+\wt\mu_3
+
(x+\alpha)(\wt\mu_2+\wt\mu_3)
+
(x+\alpha)^2[-\wt\mu_0-2\wt\mu_1 - \wt\mu_2].
\end{align*}

For $k=0$ in part (a) of Theorem \ref{t-AdMom} the recursion relation reduces to $$\wt\mu_2 = 2\alpha \wt\mu_1+\alpha\wt\mu_0,$$ so that $$-\wt\mu_0-2\wt\mu_1 = -\wt\mu_2/\alpha.$$ And recall that the recursion relation for $k=1$ reduces to $$\wt\mu_3 = (2\alpha+1)\wt\mu_2.$$

With this we have up to normalization
\[
L_2^\alpha(x)
=
(2\alpha+2)\wt\mu_2
+(x+\alpha)(2\alpha+2)\wt\mu_2
-(x+\alpha)^2 (1/\alpha +1) \wt\mu_2.
\]
Factoring out $$-\left(\frac{1}{\alpha}+1\right)\wt\mu_2 = - \frac{\alpha+1}{\alpha}\wt\mu_2$$ we obtain
\begin{align*}
L_2^\alpha(x)
&=
 - \frac{\alpha+1}{\alpha}\wt\mu_2\left[
-2\alpha-2\alpha(x+\alpha) + (x+\alpha)^2\right]
\\&
 =
  - \frac{\alpha+1}{\alpha}\wt\mu_2\left[ -2\alpha -2\alpha x -2\alpha^2 +x^2 +2\alpha x +\alpha^2\right]
 \\&
=
 - \frac{\alpha+1}{\alpha}\wt\mu_2\left[
 x^2 - \alpha^2 -2\alpha\right].
\end{align*}
This is in agreement with the degree polynomial given by \eqref{eq:list05}, since they both span the same eigenspace.

Again, we compare this to the Remark \ref{r-normalization} about normalization. It is tedious but elementary to compute
\[
\frac{\wt K_2}{\det \wt A}
=
\frac{-(\alpha+2)}{\wt\mu_2(\alpha+3+2/\alpha)}
\]
for $\wt K_2$ as in Remark \ref{r-normalization}.
Now, the leading term of $L^\alpha_2(x)$ (with the normalization described in Theorem \ref{t-AltRep} and with $\wt K_2$ from Remark \ref{r-normalization}) is given by
\[
\left(-\frac{\alpha+1}{\alpha} \wt\mu_2\right)\left( \frac{\wt K_n }{ \det \wt A}\right)
=
\frac{(\alpha+1)(\alpha+2) }{\alpha^2+3\alpha + 2} = 1.
\]
And, again, this is in agreement with the predicted leading term coefficient in Remark \ref{r-normalization}, since  for $n=2$ we have $(-1)^n/(n-1)! =1$. 
\end{ex*}

\emph{Acknowledgement.} We would like to thank K. Busse and L.L.~Littlejohn for useful discussions.

\end{document}